\documentclass[12pt]{amsart}
\usepackage{latexsym, amsmath, amssymb,amsthm,amsopn,amsfonts}
\usepackage{version}
\usepackage{epsfig,graphics,color,graphicx,graphpap}
\usepackage{amssymb}
\usepackage{todonotes}
\usepackage{listings}
\usepackage[citecolor=blue]{hyperref}

\usepackage{tikz}

\advance\hoffset by -0.5 truecm

\newcommand{\redsout}{\bgroup\markoverwith{\textcolor{red}{\rule[0.5ex]{2pt}{.4pt}}}\ULon}
\newcommand{\LC}{\left(}
\newcommand{\RC}{\right)}

\newcommand{\p}{\partial}

\numberwithin{equation}{section}
\newtheorem{theorem}{Theorem}[section]

\newtheorem{proposition}{Proposition}[section]
\newtheorem{lemma}{Lemma}[section]
\newtheorem{definition}{Definition}[section]

\newtheorem{remark}{Remark}[section]
 
\newcommand{\R}{\mathbb R}
\newcommand{\0}{{\bf 0}}

\begin{document}

\author[Lai]{Ru-Yu Lai}
\address{School of Mathematics, University of Minnesota, Minneapolis, MN 55455, USA}
\curraddr{}
\email{rylai@umn.edu }
%\thanks{The second author was supported in part by }

\author[Zhou]{Hanming Zhou}
\address{Department of Mathematics, University of California Santa Barbara, Santa Barbara, CA 93106-3080, USA}
\curraddr{}
\email{hzhou@math.ucsb.edu}

 
\title[Global determination for an inverse problem from vortex dynamics]{Global determination for an inverse problem from the vortex dynamics}
\date{\today}

\begin{abstract}
	We consider the problem of reconstructing a background potential from the dynamical behavior of vortex dipole. We prove that under suitable conditions, one can uniquely reconstruct a real-analytic potential by measuring the entrance and exit positions as well as travel times between boundary points.
	In particular, the work removes the flatness assumption on the potential from the earlier result. A key step of our method is a constructional procedure of recovering the boundary jet of the potential.
\end{abstract}

\maketitle
%\tableofcontents

\section{Introduction}
We study the reconstruction of a background potential from the dynamics of vortices. 
We consider a pair of vortices $\{a_+, a_-\}$ of opposite charge (vortex dipole) govern by the following ODE system in $\R^2$:
\begin{align}\label{dipole ode}
\left\{ \begin{array}{rclrcl}
\dot{a}_+(s)   = {1\over \pi} { (a_+ - a_-)^\perp \over |a_+ - a_-|^2 } + \nabla^\perp Q(a_+),\\ [.5em]
\dot{a}_-(s)   = {1\over \pi}  { (a_+ - a_-)^\perp \over |a_+ - a_-|^2 } - \nabla^\perp Q(a_-),
\end{array}\right.
\end{align}
where $\dot{a}_\pm(s):={d\over ds}a_\pm(s)$ and $Q:\mathbb R^2\to \mathbb R$ is the background potential, see \cite{LSSU} for more details about the system.
Here for a function $w:\mathbb{R}^2\rightarrow \mathbb{R}$ and for a vector $u=(u_1,u_2)$ in $\mathbb{R}^2$, we denote 
$$\nabla^\perp w = (\p_2w,\,-\p_1w),\qquad u^\perp=(u_2,\,-u_1),$$ respectively, where $\p_j={\p\over\p x_j}$ for $j=1,2$.
This ODE system \eqref{dipole ode} is linked to an inhomogeneous Gross-Pitaevskii equation in $\R^2$ in a critical asymptotic regime where vortices interact with both the background potential and each other \cite{MTKFCSFH, SKS, SMKS, Torres, torres2011dynamics}. This connection was rigorously proved in \cite{KMS}. Note that these vortex-vortex and vortex-potential interactions have been studied experimentally \cite{freilich2010real, neely2010observation} and numerically \cite{MKFCS}.

The main objective of our work is to study the inverse problem for the reconstruction of the background potential $Q(x)$ from the travel information of the dipole.
The uniqueness result for this inverse problem was first proved in \cite{LSSU} when the potential is sufficiently smooth and flat, indicating the path of the dipole is close to straight lines. Specifically, in \cite{LSSU}, the potential is uniquely identified by using the trajectory of the center of mass of the dipole. A reconstruction formula for the potential and numerical examples are also investigated in \cite{LSSU}.

This paper aims to release the smallness assumption on the gradient of the potential in \cite{LSSU} by considering a more general potential. Instead of using the information of the center of mass as in \cite{LSSU}, we will take measurements of travel trajectories of both vortices ($a_\pm$), including their positions of entrance and exit as well as the travel time.  
The detailed problem setting of this inverse problem is as follows.

\subsection{Problem setup and main results}
Let $\Omega\subset \mathbb R^2$ be a bounded convex open domain with smooth boundary $\p \Omega$, and $Q\in C^\infty(\mathbb R^2)$ be a background potential in $\mathbb R^2$. Let $U$ be an open neighborhood of $\Omega$, i.e. $\Omega\subset\subset U$, so that $U\setminus \Omega$ is a tubular neighborhood around $\Omega$.

For the study of the inverse problem, we propose a measurement map $\mathcal{S}$ on $(\p\Omega\times (U\setminus \Omega))\setminus \Delta$ with respect to (w.r.t.) a background potential $Q$, where $\Delta:=\{(x,x):\, x\in\p\Omega\}$ is a subset of $\p\Omega \times (U\setminus \Omega)$. The definition of $\mathcal{S}$ is discussed as follows.
For any $(x,y)\in(\p\Omega\times (U\setminus \Omega))\setminus \Delta$, we denote $a_\pm(s,x,y)$ to be the solutions of \eqref{dipole ode} with initial values 
$$a_+(0,x,y)=x,\qquad a_-(0,x,y)=y,$$ 
and, moreover, we define the function
\begin{equation}\label{eqn:def_time}
\begin{aligned}
\tau_+:\quad & (\p\Omega\times (U\setminus \Omega))\setminus \Delta & \to\quad&  [0,\infty)\cup \{\infty\}
\end{aligned}
\end{equation}
to be the first nonnegative time when the vortex $a_+(\cdot, x,y)$ exits $\Omega$.
In particular, if $a_+(\cdot, x,y)$ is trapped in $\Omega$, i.e. $a_+(s,x,y)\in \overline\Omega$ for all $s\geq 0$, then we define $\tau_+(x,y)=\infty$.  
Notice that the governing ODEs \eqref{dipole ode} become singular when the vortices $a_\pm$ collide. For the sake of simplicity, we assume that 
$$\tau_+(x,y)<\infty, \quad \forall (x,y)\in (\p\Omega\times (U\setminus \Omega)) \setminus \Delta,$$
and the dipole never collides in $\Omega$. Because of the local nature of our approach discussed in later sections, this assumption indeed does not impose any restriction to the main results of the current paper. 

Now we are ready to define the measurement map $\mathcal S$ as below:
\begin{align*}
\mathcal S:\quad (\p\Omega\times (U\setminus \Omega)) \setminus \Delta \quad \to\quad \big([0,\infty) \times \p\Omega\times (\mathbb R^2\setminus\Omega)\big)\cup \big([0,\infty)\times \p\Omega\big) 
\end{align*}
%&(x,y)&\mapsto\quad &   {\color{red}(\tau_+(x,y),\, a_+(\tau_+(x,y),x,y),\, a_-(\tau_+(x,y),x,y) )}, %A_+(x,y) & 
and
\begin{align}\label{eqn:def_S}
&\mathcal S(x,y) \notag\\
&=\left\{ \begin{array}{ll}
\big(\tau_+(x,y),\, a_+(\tau_+(x,y),x,y),\, a_-(\tau_+(x,y),x,y) \big), \quad \mbox{if}\; a_-(\tau_+(x,y),x,y)\notin \Omega,\\ [.5em]
\big(\tau_+(x,y),\, a_+(\tau_+(x,y),x,y) \big), \quad \mbox{if}\; a_-(\tau_+(x,y),x,y)\in \Omega.
\end{array}\right.
\end{align}
In particular, the definition of $\mathcal S$ says that we don't know the behavior of the dipole inside $\Omega$. As a matter of fact, we will only make use of those points $(x,y)$ such that $a_-(\tau_+(x,y),x,y)\notin\Omega$ in this paper. More specifically, the restriction of $\mathcal S$ on a subset of $(\p\Omega\times (U\setminus \Omega)) \setminus\Delta$ is sufficient for our approach to the reconstruction of the potential.
 
\begin{remark}
We make the following comment regarding the definition \eqref{eqn:def_S}:
If we define the measurement operator $\mathcal S$ on $(\p\Omega\times \p\Omega )\setminus \Delta$, instead of the domain in \eqref{eqn:def_S}, then the initial velocities $\dot a_\pm(0,\cdot,\cdot)$ might not cover all the directions. 
For example, let $\Omega$ be a disk and $x\in \p \Omega$ be a fixed point and also let $Q\equiv 0$. It is clear to see that the velocities $\dot a_\pm(0,x,y)$ are never normal to $\p \Omega$ for any $(x,y)\in (\p\Omega\times \p\Omega)\setminus \Delta$. 
Moreover, in this paper we rely on those (local) trajectories which are almost tangent to the boundary to identify the behavior of $Q$ near $\p\Omega$. These local trajectories are available only when $x$ and $y$ are far away from each other on $\p\Omega$, since in this case they can make $\dot a_\pm(0,x,y)$ almost tangent to $\p\Omega$. However, in practice this could be difficult to manipulate when the domain $\Omega$ is large.

\end{remark}

The inverse problem under study is the determination of the potential $Q$ in $\Omega$ from measurements $\mathcal S$. 
As mentioned above, the unique determination result of a background potential from the trajectories of vortex dipole was studied in \cite{LSSU}, under the assumption that the potential is almost flat in a suitable space. The measurements taken in \cite{LSSU} are regarding the Hamiltonian flow of the dipole center, i.e. $(a_++a_-)/2$, while in the current paper, we measure the trajectories of $a_+$ and $a_-$ separately without knowing the phase (velocity vector) information. In other words, the inverse problem under consideration only uses phaseless data, which is more applicable in practice.

The main strategy consists of two steps - local determination and global determination. We first show that all the derivatives of the potential $Q$ at a boundary point on $\p\Omega$, i.e. the boundary jet of $Q$, can be reconstructed from $\mathcal{S}$, see Theorem~\ref{boundary determination}. Then this established local result yields the global reconstruction of $Q$ in Theorem~\ref{real analytic case} due to the analytic assumption of $Q$.

To achieve this goal, we need an assumption on the convexity of the boundary w.r.t. the potential which is stated in the following definition.  
\begin{definition}\label{def:intro}
	 Let $p\in \p \Omega$, we say that the boundary $\p\Omega$ is {\it strictly convex at $p$ w.r.t. the background potential $Q$} if there exists a small open neighborhood $V_p$ of $p$, such that $a_+(s,p,q)\notin \overline\Omega$, $s\in (-\delta, \delta)\setminus\{0\}$ for some $\delta>0$, whenever $q\in V_p\setminus\Omega$, $q\neq p$, satisfying $\dot a_+(0,p,q)$ tangent to $\p\Omega$. 
\end{definition}
For example, given a bounded domain $\Omega$ with strictly convex boundary (w.r.t. the Euclidean metric) and $p\in\p\Omega$, we assume that $|\nabla Q(p)|$ is sufficiently small, then $\p\Omega$ is strictly convex at $p$ w.r.t. $Q$ (viewed as a small perturbation of trivial potential).
Note that even though we define the convexity by looking at the trajectory $a_+$, Definition~\ref{def:intro} could also be defined in terms of $a_-$ due to the symmetry of $a_+$ and $a_-$ in the ODEs \eqref{dipole ode}.

We have the following first theorem which states that the derivatives of the potential on the boundary $\p\Omega$ can be determined from $\mathcal S$ in a local manner.
\begin{theorem}[Boundary determination]\label{boundary determination}
Let $Q\in C^\infty(\mathbb R^2)$ be a background potential and $\Omega\subset \mathbb R^2$ be a bounded convex open domain with smooth boundary. Suppose that $\p \Omega$ is strictly convex w.r.t. $Q$ at $p\in \p\Omega$. There exists an open neighborhood $\widetilde\Omega$ of $\overline\Omega$.  
Suppose that $Q$ is known in $\mathbb R^2\setminus \widetilde \Omega$,  
then there exists an open neighborhood $W$ of $p$, $W\setminus \overline{\widetilde \Omega}\neq \emptyset$, so that the measurement $\mathcal S$ restricted in $W\setminus \Omega$ determines 
$$\p^\alpha Q(p) \qquad \hbox{ for all multi-indices $\alpha$ with }|\alpha|\geq 1.$$
\end{theorem}

%------------------------------------------------------------------------------
\begin{figure}[ht]\label{dipole figure}
\begin{tikzpicture} [scale=.65]
% Domain \Omega 
\draw (2,0) circle [radius=4.5];
\draw (0,0) to [out=90, in=180] (3,2);
\draw (3,2) to [out=0, in=20] (2,-2);
\draw (2,-2) to [out=190, in=-5] (1,-2);
\draw (1,-2) to [out=170, in=270] (0,0);

% Domain \wildtilde\Omega 
\draw (-.9,0) to [out=90, in=180] (3.3,2.9 );
\draw (3.3,2.9 ) to [out=0, in=20] (2.8,-2.6);
\draw (2.8,-2.6) to [out=205, in=-8] (1,-2.8);
\draw (1,-2.8) to [out=170, in=270] (-.9,0);

% Path of a_+,  
\draw [->,>=latex, dotted, thick] (.41,-1.7) to [out=10, in=170] (2,-2);
% Path of a_-,  
\draw [->,>=latex, dotted, thick] (.25,-3) to [out= 20, in=180] (2.2,-3.3);

% Notations 
\filldraw  
(.41,-1.7) circle (1.5pt) node[below,font=\tiny] {$a_+$};
\filldraw
(.25,-3) circle (1.5pt) node[below,font=\tiny] {$a_-$};

\filldraw  
(2,-2) circle (1.5pt) node[below,font=\tiny] {$a_+(\tau_+)$};
\filldraw
(2.2,-3.3) circle (1.5pt) node[below,font=\tiny] {$a_-(\tau_+)$};

%%%%%%%%%%%%%%%
\draw (2,0) node[above,font=\large] {$ {\Large \Omega}$};

\draw (2,1.8) node[above] {$\widetilde\Omega$};

\draw (-.5,2.3) node[above,font=\large] {$U$};
\end{tikzpicture}
\caption{ 
	The figure illustrates the travel path of a dipole ($a_+$ with positive charge and $a_-$ with negative charge) with initial positions in $U\setminus\Omega$. The vortex $a_+$ reaches the boundary of domain $\Omega$ at the exit time $\tau_+$.
	 } 
\end{figure}
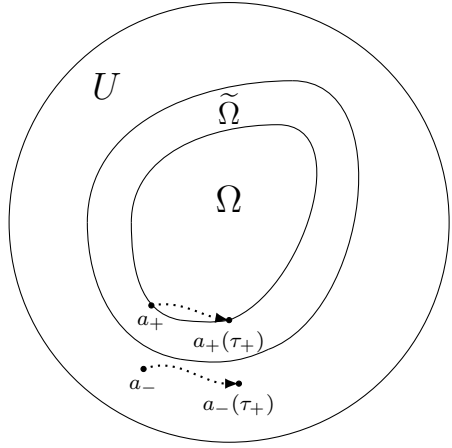
%------------------------------------------------------------------------------

The existence of $\widetilde \Omega$ is analyzed in Section~\ref{sec:tildeOmega}. Roughly speaking, the thickness of the tubular neighborhood $\widetilde \Omega\setminus \Omega$ has an upper bound depending on $\sup_{x\in\p \Omega} |\nabla Q(x)|$ and the size of the small open neighborhood $V_p$ in the Definition \ref{def:intro}. In particular, we can choose the open neighborhood $W$ from Theorem \ref{boundary determination} to be the same as $V_p$.  
Notice that we do not know the information of $Q(x)$ for $x\in \widetilde{\Omega}\setminus\overline\Omega$, otherwise if $Q$ is given in $\widetilde{\Omega}\setminus\overline\Omega$, then one can determine all the derivatives of $Q$ on $\p \Omega$ by simply taking limits from outside of $\Omega$.

It is worth emphasizing that the determination of the boundary jet of $Q$ from $\mathcal S$ is purely local. More specifically, to determine $\p^\alpha Q(p)$ for all $|\alpha|\geq 1$, one only needs the information of the measurements
$$\{\mathcal S(x,y): (x,y)\in (W\cap \p\Omega)\times (W\setminus \Omega), x\neq y\}\cap ([0,\delta) \times \p\Omega\times (W\setminus\Omega))$$ 
for some small constant $0<\delta\ll 1$ and a small open neighborhood $W$ of $p\in\p\Omega$ so that $W\setminus \overline{\widetilde \Omega}\neq \emptyset$.  

Theorem \ref{boundary determination} immediately implies the global determination of real-analytic background potentials.  
Notice that a convex domain is path-connected.

\begin{theorem}[Global determination]\label{real analytic case}
Let $Q\in C^\infty(\mathbb R^2)$ be a background potential and let $\Omega\subset \mathbb R^2$ be a bounded convex open domain with smooth boundary. Suppose that $\p \Omega$ is strictly convex w.r.t. $Q$ at some point on $\p\Omega$. There exists an open neighborhood $\widetilde\Omega$ of $\overline\Omega$.  
Suppose, in addition, that $Q$ is known in $\mathbb R^2\setminus \widetilde \Omega$ and real-analytic in $\widetilde \Omega$.  
Then for any open neighborhood $U$ of $\overline{\widetilde\Omega}$, the measurement $\mathcal S$ in $U\setminus\Omega$ determines $Q$ in $\widetilde\Omega$.
\end{theorem}

\subsection{Methodology}
In this paper, we are motivated by the techniques developed in the study of classical geometric inverse problem, namely the {\it boundary rigidity problem}, which consists of determining a Riemannian metric on a compact smooth manifold with boundary from the collection of distance (w.r.t. the metric) between pairs of boundary points. This problem arises naturally from the geophysical question of recovering the inner structure of the Earth from the travel time of seismic waves. We refer interested readers to the survey paper \cite{SUVZ19} and the references therein for recent developments of the boundary rigidity and related problems.

A key ingredient in the proof of Theorem \ref{boundary determination} is an integral identity first derived by Stefanov and Uhlmann \cite{SU3} in the study of the boundary rigidity problem for almost flat metrics. The integral identity works for general flows given by the solutions of governing ODE system. In our case, it connects the difference of measurement operators of two potentials $Q_1$ and $Q_2$ with some weighted integral of the difference of Hamiltonian vector fields of $Q_1$ and $Q_2$, see \eqref{F2int2}. In particular, we let $Q_2=0$, the trivial potential, and repeatedly differentiate the integral to recover all the derivatives of $Q_1$ at a boundary point. Notice that our approach of the boundary determination of the potential is constructional. The integral identity has also been used in the study of the boundary  and global determination questions in various inverse problems, e.g. the boundary rigidity problem \cite{CQUZ, SUV, SUV17,UW03, UYZ20, W}, the non-abelian Radon transform \cite{PSUZ, Zh17} and the lens rigidity problem for Yang-Mills potentials \cite{PUZ19,Zh18}.

The paper is structured as follows. In Section~\ref{sec:preliminary}, we discuss the construction of the domain $\widetilde{\Omega}$ and the preliminary setting for the problem under study. Section~\ref{sec:local reconstruction} is devoted to recovering the derivatives of the potential at a given boundary point by applying the Stefanov-Uhlmann identity. Then in Section~\ref{sec:global reconstruction}, the global determination of the potential follows by using the local result in Section~\ref{sec:local reconstruction} and the analyticity of the potential. Finally, additional useful lemmas are placed in the Appendix~\ref{sec:appendix}.

\section{Some preliminary analysis}\label{sec:preliminary}
In this section, we first show that the domain $\widetilde{\Omega}\supset\Omega$ exists so that one can choose the initial position of $a_-$ outside of $\widetilde{\Omega}$ to make the initial velocity of $a_+$ tangent to the boundary $\p\Omega$. Then based on the assumption that $Q$ is known in $\R^2\setminus\widetilde{\Omega}$, it implies $\nabla^\perp Q(a_-(0))$ is known. Next, with the help of the boundary normal coordinate, we show that the exit time of $a_+$ in $\Omega$ satisfies a useful property stated in Lemma~\ref{derivative of exit time}.
 
	\subsection{Construction of $\widetilde{\Omega}$}\label{sec:tildeOmega}
	Fixing some boundary point $p\in\p\Omega$ and since $\Omega$ is convex, by shifting or rotating it, one can suppose that the $x_1$-axis is tangent to $\p\Omega$ at $p$ and the $x_2$-axis is normal to $\p\Omega$ at $p$. Moreover, we require $\p_{2}:={\p\over \p x_2}$ to be inward pointing so that $\Omega\subset \{x_2>0\}$ (note that $\Omega$ is convex). From now on, we will use the notation $\0$ to denote the zero in coordinates and use $0$ to denote the scalar zero. Without loss of generality, we now suppose $p=\0$.
	
	We consider the following evolution equations of a dipole $\{a_+, a_-\}$ in the coordinates with initial position $\0$ and $\xi$, respectively: 
	\begin{align}\label{dipolepair 2}
	\left\{\begin{array}{l} 
	\dot{a}_+(s) = {1\over \pi} { (a_+ - a_-)^\perp \over |a_+ - a_-|^2 } + \nabla^\perp Q(a_+),\\ [.5em]
	\dot{a}_-(s) =  {1\over \pi}  { (a_+ - a_-)^\perp \over |a_+ - a_-|^2 } - \nabla^\perp Q(a_-),\\ [.5em]
	a_+(0) = \0,\\ [.5em]
	a_-(0) = \xi,\\ 
	\end{array}\right.
	\end{align}
	where $\xi=(\xi_1,\xi_2)\notin \Omega$ and $\xi\neq \0$.
	We are interested in the dependence of the initial velocity $\dot a_+(0)$ on the initial value $a_-(0)$  
	and also require $\dot a_+(0)$ to be almost tangential to $\p\Omega$ in order to address the boundary determination question in later section.

	To this end, we first write $a_\pm=(a^1_\pm, a^2_\pm)$, then the first equation of \eqref{dipolepair 2} can be rewritten in coordinates as
	\begin{align}\label{evolution local}
	\dot a^1_+(0) =\frac{-\xi_2}{\pi |\xi|^2}+\p_2 Q(\0), \quad \dot a^2_+(0) =\frac{\xi_1}{\pi |\xi|^2}-\p_1 Q(\0).
	\end{align}
	Here $|\xi|$ is the Euclidean distance between $p=\0$ and $\xi$. 
	
	Next we assume that $\dot a_+^2(0)=0$, i.e. $\dot a_+(0)\neq \0$ is tangent to $\p\Omega$, then \eqref{evolution local} gives 
	$$\frac{\xi_1}{\pi |\xi|^2}=\pi \p_1 Q(\0)  .$$
	We split the discussion into the following three cases:
	\begin{enumerate}
		\item When $\p_1 Q(\0)=0$, we have ${\xi_1\over \pi|\xi|^2}=0$ which implies $\xi_1=0$ for any $\xi_2<0$. Then
		 we can choose almost any $\xi_2<0$ such that 
		$$\dot a_+^2 (0)=0, \quad \dot a_+^1(0)\neq 0 \hbox{ and $\xi=(0,\xi_2)\notin \Omega$}.$$ 
		\item When $\p_1 Q(\0)> 0$, we have $\xi^2_2=\xi_1(\frac{1}{\pi \p_1 Q(\0)}-\xi_1)$. Then we can choose almost any $0<\xi_1<\frac{1}{\pi \p_1 Q(\0)}$ such that $$\dot a_+^2 (0)=0, \quad \hbox{$\dot a_+^1(0)\neq 0$} \hbox{ and }\hbox{$\xi=(\xi_1,\xi_2)\notin \Omega$ (with $\xi_2<0$)}. 
		$$ 

		\item The case $\p_1 Q(\0)<0$ leads to a similar conclusion as case (2).
	\end{enumerate}
In particular, in cases $(2)$-$(3)$ (that is, $\p_1 Q(\0)\neq 0$), if $0<|\xi_1|<\frac{1}{\pi |\p_1 Q(\0)|}$, then $\frac{-1}{2\pi |\p_1 Q(\0)|}\leq \xi_2<0$. Moreover, if we have the upper bound $\sup_{x\in\p\Omega}|\nabla Q(x)|\leq M$ for some $M>0$, then one can always make $\frac{-1}{2\pi M}<\xi_2<\frac{-1}{4\pi M}<0$, 
independent of $p\in \p\Omega$. We note that for case (1), this bound for $\xi_2$ is straightforward.
	This implies the existence of an open neighborhood $\widetilde \Omega$ of $\Omega$ which is defined as follows: 
	$$\widetilde \Omega:=\{x\in\mathbb R^2\; |\; \mbox{dist}\,(x, \Omega)< \min\; \{ (4\pi M)^{-1}, \sigma\}\}$$
    with $0<\sigma\ll 1$ a fixed number so that $V_p\setminus \overline{\widetilde\Omega}\neq \emptyset$. Here $V_p$ is the open neighborhood appearing in Definition \ref{def:intro} of the convexity at $p=\0\in \p\Omega$ w.r.t. $Q$. 
	In other words, $\widetilde \Omega\setminus \Omega$ is a tubular neighborhood around $\Omega$. 
	Therefore for all three cases, we can always find a vector $\xi\notin \overline{\widetilde \Omega}$, so that $\dot a_+(0)\neq \0$ is tangent to $\p\Omega$. By continuity, the existence of $\xi\notin \overline{\widetilde \Omega}$ is also valid in order to have almost tangential $\dot a_+(0)$.

\subsection{Behavior near a convex boundary point}\label{sec:boundary jet}

Suppose that the dipole $\{a_+, a_-\}$ with initial position $\0$ and $\xi(t)$, respectively, satisfies the following problem:
\begin{align}\label{dipolepair}
\left\{\begin{array}{l} 
\dot{a}_+(s;t) = {1\over \pi} { (a_+ - a_-)^\perp \over |a_+ - a_-|^2 } + \nabla^\perp Q(a_+),\\ [1em]
\dot{a}_-(s;t) =  {1\over \pi}  { (a_+ - a_-)^\perp \over |a_+ - a_-|^2 } - \nabla^\perp Q(a_-),\\ [1em]
a_+(0;t) = \0,\\ [1em]
a_-(0;t) = \xi(t).\\ 
\end{array}\right.
\end{align}

Section~\ref{sec:tildeOmega} suggests that there exists an open bounded domain $\widetilde{\Omega}$ containing $\overline\Omega$ so that the following is valid: There exists nontrivial initial data $\xi_0\notin  \overline{\widetilde \Omega}$ of $a_-$ such that $\dot a_+(0;0)$ is tangent to $\p\Omega$.  
Moreover, by continuity, there exists a smooth curve 
$$
\xi (\cdot):[0,\delta)\to \R^2\setminus \overline{\widetilde\Omega} 
$$ for small $0<\delta\ll1$ so that $\dot a_+(0;t)$ is almost tangent to $\p\Omega$ and $\dot a_+^2(0;t)>0$ (i.e. inward pointing) for $t\in (0,\delta)$.

The convexity of $\Omega$ w.r.t. $Q$ at $p=\0$ yields that for each $t\in (0,\delta)$, $a_+(\cdot;t)$ exits $\Omega$ at some point 
$c(t)\in\p\Omega$ which is close to $p$ and satisfies $$c(t)\to p  \qquad \hbox{ as }t\to 0.$$ Moreover, the exit time $\ell(t)$, i.e. $a_+(\ell(t);t)=c(t),$ is smooth and satisfies $$\ell(t)\to 0  \qquad\hbox{ as } t\to 0,$$ 
while $a_-(s;t)\in \R^2\setminus \widetilde\Omega$  
for $s\in [0,\ell(t)]$. Note that for the simplicity of the notation, here we denote $\ell(t):=\tau_+(\0,\xi(t))$ where $\tau_+(\0,\xi)$ is defined in \eqref{eqn:def_time}.
It is worth pointing out that $c(t)$ and $\ell(t)$ can be determined from the measurements $\mathcal S(\0,\xi(t))$.

\subsubsection{Boundary normal coordinates.}

To connect the behavior of the dipole with the measurement $\mathcal S$ near the boundary, it is helpful to introduce the {\it boundary normal coordinates}. 

Consider a neighborhood $W$ of a boundary point $p\in\p\Omega$, equipped with the boundary normal coordinates $\{z^1, z^2\}$, so that 
$$\p\Omega\cap W=\{z^2=0\},\qquad W\cap \Omega\subset\{z^2>0\},$$ 
and $z^1=c$ ($c$ is a constant) are straight lines normal to the boundary $\p \Omega$. Then the boundary near $p$ is straightened in such coordinates. Moreover, in the boundary normal coordinates, the Euclidean metric takes the form 
$$e=f(z^1,z^2) (dz^1)^2+(dz^2)^2.$$  
Here $f$ is a positive (local) function, $f(\0)=1$.

Let's also take $p=\0$ in this coordinates. Notice that $\Omega$ is a convex domain, one can always make the neighborhood $W$ large enough outside $\Omega$ so that there exists $\xi\in W\setminus \overline{\widetilde \Omega}\neq \emptyset$ with $\dot a_+(0)$ tangent to $\p\Omega$ at $\0$.
Suppose that $\p\Omega$ is strictly convex w.r.t. $Q$ at $p=\0$. In the boundary normal coordinates, the definition of the strict convexity of $\p\Omega$ w.r.t. $Q$ at $\0$ can be interpreted as follows: Let $\{a_+,a_-\}$ be solutions of \eqref{dipolepair 2} with $\dot a_+(0)$ tangent to $\p\Omega$, then in the boundary normal coordinates, one has
$$\dot a^2_+(0)=0,\quad \ddot a^2_+(0)<0.$$

In the lemma below, we will show the exit time $\ell(t)$ is increasing locally at $t=0$ with the help of the boundary normal coordinate.
\begin{lemma}\label{derivative of exit time}
There exists $\xi (\cdot):[0,\delta)\to W\setminus \overline{\widetilde\Omega}$ with $\xi(0)=\xi_0$, such that $\ell'(0)>0$.
\end{lemma}
\begin{proof}
We now denote $a_+(s;t)=:(a_+^1(s;t), a_+^2(s;t))$ in the boundary normal coordinates.
By continuity, there exists $\xi(\cdot):[0,\delta)\to W\setminus \overline{\widetilde\Omega}$ with $\xi(0)=\xi_0$, a constant $\varepsilon>0$, such that
\begin{equation}\label{choice of xi(t)}
\dot{a}_+(0;t)=(\dot{a}_+^1(0;t),\,\dot{a}_+^2(0;t))=\varepsilon (\alpha(t),\, \beta t), \quad \beta>0, 
\end{equation}
where $\alpha(t)^2+\beta^2t^2 =1$. In particular, if $t=0$, then 
$
|\alpha(0)|=1.
$
By the strict convexity w.r.t. $Q$ at $p$, it is clear that $\ell(0)=0$. 

To show $\ell'(0)>0$, we first show $\ell'(0)\neq 0$ by applying the contradiction argument. Suppose that $\ell'(0)= 0$. We can then write $\ell(t)$ asymptotically for small $t$ as 
		\begin{align}\label{ell expansion}
		\ell(t) =\ell(0)+\ell'(0)t+O(t^2)= O(t^2).
		\end{align}
		When $s$ is sufficiently small, we have the asymptotic expansion of $a_+$ as follows:
		$$
		a_+(s;t)= a_+(0;t) +\dot{a}_+(0;t) s +O(s^2).
		$$
		We note that the second component $a_+^2(\cdot;t)$ always vanishes at $s=\ell(t)$ in the boundary normal coordinates (the boundary is straightened near $p$), namely,
		$$
		0\equiv a_+^2(\ell(t);t)=\dot{a}_+^2(0;t) \ell(t) +O(\ell(t)^2),
		$$
		and, moreover, we then have
		\begin{align}\label{ell expansion 2}
		0= \varepsilon\beta t\,\ell(t) +O(\ell(t)^2) 
		\end{align}
		for $t\in[0,\delta)$ for some small $\delta>0$.
		Since $\beta>0$ and \eqref{ell expansion}, it implies that \eqref{ell expansion 2} can not hold for all $t$ in $[0,\delta)$. This leads to a contradiction and thus $\ell'(0)\neq 0$.	Finally since $\ell(t)>0$ for $t>0$, we obtain that $\ell'(0)>0$.	
\end{proof}

\begin{remark} 
Lemma \ref{derivative of exit time} shows the existence of $\xi(t)$ with $\ell'(0)\neq 0$. As one will see in Section \ref{sec:1 derivative}, this is sufficient for determining $\nabla Q(\0)$ from $\mathcal S$. Then using the ODEs \eqref{dipolepair}, we can determine the relation between $\xi(t)$ and $\dot a_+(0;t)$. In particular, we are able to choose $\xi(t)$ so that $\ell'(0)\neq 0$, and $\dot a_+(0;t)$ exactly has the form as in \eqref{choice of xi(t)}.
\end{remark}

%%%%%%%%%%%%%%%%%%%%%%%%%%%
%%%%%% Local Result %%%%%%%
%%%%%%%%%%%%%%%%%%%%%%%%%%%
\section{Determination of the boundary jet}\label{sec:local reconstruction} 
In this section, we will show the local reconstruction of $Q$ from the measurement $\mathcal S$. 
This section consists of two main parts. In Section~\ref{sec:SU identity}, we first introduce the Stefanov-Uhlmann identity. This identity connects the data $\mathcal{S}$ to the discrepancy in the two vortex dynamic in a given potential (we will take the given one to be trivial potential $Q_0$) and the unknown potential $Q$, respectively. Then we will apply this identity to recover such potential $Q$ in Section~\ref{sec:jet}.

\subsection{Stefanov-Uhlmann identity}\label{sec:SU identity}
Following the setting at the beginning of Section~\ref{sec:boundary jet}, now we denote the initial boundary data by 
$$
\phi(t):= (\0,\, \xi_t),\qquad \xi_t:=\xi(t) \quad \hbox{for small }t.
$$
We denote the trajectory of the dipole in the background $Q$ by 
$$
X(s,\phi(t)) = (a_+(s;t),\,a_-(s;t)).
$$
Then $X(0,\phi(t))=\phi(t)$.

When $Q\equiv 0$ (we denote the trivial potential by $Q_0$), we particularly use the notation 
$$X_0(s,\phi(t))=(a^0_+(s;t),\,a^0_-(s;t))$$
to denote the trajectory of the dipole in this trivial background. Specifically, from the ODEs \eqref{dipolepair}, one can derive the exact expression of $a^0_\pm(s;t)$ as follows:
$$
a^0_+(s;t)= {-\xi_t^\perp \over \pi|\xi_t|^2} s,\qquad a^0_-(s;t)= a^0_+(s;t)+\xi_t.
$$
Note that since $X_0(s,\phi(t))$ is the path of the dipole in the trivial background, we actually know its trajectory at any time $s$. Then $X_0(\ell(t),\phi(t))$ is indeed a known data.

The Stefanov-Uhlmann integral identity, derived in \cite{SU3}, is adjusted in our setting and is read as follows:
\begin{proposition}[Stefanov-Uhlmann Identity]\label{Prop:SUid} We have
\begin{align} \label{F2int2}
&X(\ell(t),\phi(t))-X_0(\ell(t),\phi(t)) \notag \\
& =\int^{\ell(t)}_0 {\partial X_0 \over \partial \phi(t)} (\ell(t)-s, X(s, \phi(t)))(V - V_0)(X(s,\phi(t)))\,ds,
\end{align}	 
where the velocity matrix in the background potential $Q$ is defined by
\begin{equation*}\label{definition V}
\begin{split}
V(X(s,\phi(t))) &:= (\dot a_+(s;t), \; \dot a_-(s;t))^T,  
%&= \LC {1\over \pi} { (a_+ - a_-)^\perp \over |a_+ - a_-|^2 } + \nabla^\perp Q(a_+),\, {1\over \pi} { (a_+ - a_-)^\perp \over |a_+ - a_-|^2 } - \nabla^\perp Q(a_-)\RC  ,
\end{split}
\end{equation*}
and when $Q\equiv 0$, we denote
\begin{equation*}\label{definition V_0}
\begin{split}
V_0(X(s,\phi(t))) &:= \bigg(\frac{(a_+(s;t)-a_-(s;t))^\perp}{\pi |a_+(s;t)-a_-(s;t)|^2},\; \frac{(a_+(s;t)-a_-(s;t))^\perp}{\pi |a_+(s;t)-a_-(s;t)|^2}\bigg)^T. 
\end{split}
\end{equation*} 
Here the superscript $T$ denotes the transpose of a vector. In particular
$$(V-V_0)(X(s,\phi(t)))=(\nabla^\perp Q(a_+(s;t)),\; -\nabla^\perp Q(a_-(s;t)))^T.$$
\end{proposition} 
 
Since $X(\ell(t),\phi(t))$ for $t$ small is given by the measurement $\mathcal S$, we have $X(\ell(t),\phi(t))-X_0(\ell(t),\phi(t))$ is known, which implies that the right-hand side of \eqref{F2int2} is also known for sufficiently small $t$. From now on, we denote the integral in \eqref{F2int2} by
\begin{align} \label{F2int1}
R(t):=\int^{\ell(t)}_0 {\partial X_0 \over \partial \phi(t)} (\ell(t)-s, X(s, \phi(t)))(V - V_0)(X(s,\phi(t)))\,ds.
\end{align}

In the remaining of this section, we will discuss how to obtain the information of $Q$ near $p=\0$ from this known data $R(t)$.
The key strategy is taking the derivative of $R(t)$ w.r.t. $t$ multiple times in order to extract useful information of the derivatives of $Q$. Let's denote the $k$-th derivative of $R(t)$ by $R^{(k)}(t)$, that is,
$$R^{(k)}(t):={d^k\over dt^k} R(t).$$

\subsection{Reconstruction of boundary jet}\label{sec:jet}
We will recover all the derivatives of the potential $Q$ at boundary point $\0$ by the induction argument. In particular, we will discuss the derivative of $R(t)$ up to fifth order in details, which will provide crucial observations regarding the reconstruction of the higher order term of $Q$.

\subsubsection{\bf $1^{st}$ derivative w.r.t. $t$}\label{sec:1 derivative}
We differentiate \eqref{F2int1} w.r.t. $t$ at $t=0$. Then we obtain
\begin{align*} 
\lim\limits_{t\rightarrow 0}R^{(1)}(t)&=\lim\limits_{t\rightarrow 0}\Big[ \ell'(t) {\partial X_0 \over \partial \phi } (0, X(\ell(t), \phi(t)))(V - V_0)(X(\ell(t),\phi(t)))  \\
&\quad + \int^{\ell(t)}_0 {\p\over \p t}\LC  {\partial X_0 \over \partial \phi} (\ell(t)-s, X(s, \phi(t)))(V - V_0)(X(s,\phi(t)))\RC\,ds\Big]\\
&= \ell'(0){\partial X_0 \over \partial \phi} (0,\phi(0)) (V - V_0)(\phi(0)) \\
&= \ell'(0)\left[ \begin{array}{c}
\nabla^\perp Q(a_+(\ell(0);0))\\
-\nabla^\perp Q(a_-(\ell(0);0))\\
\end{array}\right]_{4\times 1},
\end{align*}
where we used that
$$
\ell(0)=0,\quad {\partial X_0 \over \partial \phi} (0,\phi(0)) = Id_{4\times 4}.
$$
Here $Id_{4\times 4}$ is the $4\times 4$ identity matrix.
Note that $a_+(\ell(0);0)=a_+(0;0)=\0$ and from the calculation of $\lim\limits_{t\rightarrow 0}R^{(1)}(t)$ above, we can determine the value 
$$
 \ell'(0)\left[ \begin{array}{c}
\nabla^\perp Q(\0)\\
-\nabla^\perp Q(\xi_0)\\
\end{array}\right]_{4\times 1}.
$$
Since $\ell'(0)>0$ is known from the measurement $\mathcal S$, it gives the recovery of $\nabla Q(\0)$ by considering only the first component in the above matrix. Note that all the derivatives of $Q$ at $\xi_0$ is known due to the assumption that $Q$ is known in $U\setminus \widetilde\Omega$.

\subsubsection{\bf $2^{nd}$ derivative w.r.t. $t$}\label{sec:2 derivative}
In order to determine the second order derivative of $Q$ at $\0$, we further calculate the derivative of $R^{(1)}(t)$ w.r.t. $t$ as below:
\begin{align*} 
R^{(2)}(t)& ={d\over dt}\LC\ell'(t) {\partial X_0 \over \partial \phi } (0, X(\ell(t), \phi(t)))(V - V_0)(X(\ell(t),\phi(t)))\RC\\
&\quad+\ell'(t){\p \over \p t }\LC  {\partial X_0 \over \partial \phi} (\ell(t)-s, X(s, \phi(t)))(V - V_0)(X(s,\phi(t)))\RC\Big|_{s=\ell(t)} \\
& \quad+  \int^{\ell(t)}_0 {\p^2\over \p t^2}\LC  {\partial X_0 \over \partial \phi} (\ell(t)-s, X(s, \phi(t)))(V - V_0)(X(s,\phi(t)))\RC\,ds \\
&=:R^{(2)}_1(t)+R^{(2)}_2(t)+R^{(2)}_{\text{int}}(t).
\end{align*}
It is clear that $R^{(2)}_{\text{int}}(0)=\0$. 	
Notice that if $\p\Omega$ is strictly convex w.r.t. $Q$ at $\0$, then it is strictly convex w.r.t. $Q$ for boundary points sufficiently close to $\0$. By applying the argument in Section \ref{sec:1 derivative} to the nearby boundary points, we can determine $\nabla Q(x)$ for $x\in\p\Omega$ sufficiently close to $p=\0$. Since the position information $X(\ell(t),\phi(t)$) is known from $\mathcal S$ and, in particular, $a_+(\ell(t);t)\in\p\Omega$ are close to $\0$ for small $t$, it implies that the value of $(V - V_0)(X(\ell(t),\phi(t)))$ is then known for small $t$. This gives that the term $R^{(2)}_1(t)$ is also known for small $t$ by noting that ${\p X_0\over \p\phi}$ is known. 

Now for $R^{(2)}_2$, we have
\begin{align*} 
R^{(2)}_2(t)
%&= \ell'(t){\p \over \p t }\LC  {\partial X_0 \over \partial \phi} (\ell(t)-s, X(s, \phi(t)))(V - V_0)(X(s,\phi(t)))\RC\Big|_{s=\ell(t)}\\
&= \ell'(t) {\p \over \p t}\LC{\partial X_0 \over \partial \phi} (\ell(t)-s, X(s, \phi(t)))\RC\Big|_{s=\ell(t)} (V - V_0)(X(\ell(t),\phi(t)))\\
&\quad+ \ell'(t)  {\partial X_0 \over \partial \phi} (0, X(\ell(t), \phi(t)))[\p_t(V - V_0)(X(s,\phi(t)))]\Big|_{s=\ell(t)}.
\end{align*}
Similarly, ${\p \over \p t}{\partial X_0 \over \partial \phi} (\ell(t)-s, X(s, \phi(t))|_{s=\ell(t)}$ and $(V - V_0)(X(\ell(t),\phi(t)))$ are known for $t$ small, as a result, we only need to focus on the second term in $R^{(2)}_2(t)$, that is,
\begin{align}\label{to be recovered term in 2 derivative}
    \mathcal{K}_2(t)&:=( \mathcal{K}_{2,1}(t),\, \mathcal{K}_{2,2}(t))^T  \notag\\
    &:=\ell'(t)  {\partial X_0 \over \partial \phi} (0, X(\ell(t), \phi(t)))[\p_t(V - V_0)(X(s,\phi(t)))]\Big|_{s=\ell(t)},
\end{align}
whose first term $\mathcal{K}_{2,1}(t)$ turns out to be zero as $t\rightarrow 0$ by Lemma~\ref{lemma limit V} in the Appendix. Its second term $\mathcal{K}_{2,2}(t)$ contains known data $\p_t(\nabla^\perp Q)(a_-(s;t))|_{s=\ell(t)}$ as $t\rightarrow 0$.

In conclusion, we have seen that taking the second derivative of $R(t)$ does not lead to additional information about $Q$, which motivates us to consider the next derivative, that is, $R^{(3)}(t)$.

\subsubsection{\bf $3^{rd}$ derivative w.r.t. $t$}
We recall that in $R^{(2)}(t)$, every term is known now for small $t$, except $R^{(2)}_{\text{int}}(t)$ and $\mathcal{K}_2(t)$ defined in \eqref{to be recovered term in 2 derivative} which are only known as $t\rightarrow 0$. Then after straightforward computations, we have 
\begin{align*} 
R^{(3)}(t)&=\hbox{known terms}+ {d\over d t}\LC\ell'(t) {\partial X_0 \over \partial \phi } (0, X(\ell(t), \phi(t)))  \p_t (V - V_0)(X(s,\phi(t)))  \Big|_{s=\ell(t)} \RC\\
&\quad+\ell'(t){\p^2 \over \p t^2}\LC  {\partial X_0 \over \partial \phi} (\ell(t)-s, X(s, \phi(t)))(V - V_0)(X(s,\phi(t))) \RC\Big|_{s=\ell(t)}  \\
& \quad+  \int^{\ell(t)}_0 {\p^3\over \p t^3}\LC  {\partial X_0 \over \partial \phi} (\ell(t)-s, X(s, \phi(t)))(V - V_0)(X(s,\phi(t)))\RC\,ds \\
&=:\hbox{known terms}+R^{(3)}_1(t)+R^{(3)}_2(t)+R^{(3)}_{\text{int}}(t).
\end{align*}
Note that these known terms have been known for small $t$ and the remaining terms are derived from $R^{(2)}(t)$ through
$$R^{(3)}_1(t):={d\over dt}\mathcal{K}_2(t)\qquad \hbox{ and } R^{(3)}_2(t)+R^{(3)}_{\text{int}}(t) := {d\over dt} R^{(2)}_{\text{int}}(t).$$
Again, it is clear to see that $R^{(3)}_{\text{int}}(0)=\0$. Moreover, we can also derive that $\lim\limits_{t\rightarrow 0}R^{(3)}_2(t)$ is indeed known by Lemma~\ref{lemma limit V}.

For the term $R^{(3)}_1(t)$, Lemma~\ref{lemma limit V} is applied to get that
\begin{align*}
 \lim\limits_{t\rightarrow 0} R^{(3)}_1(t) &=\hbox{known terms}\\
&\quad+ \lim\limits_{t\rightarrow 0} \ell'(t) {\partial X_0 \over \partial \phi } (0, X(\ell(t), \phi(t)))  {d\over dt} \left[ \p_t (V - V_0)(X(s,\phi(t)))  |_{s=\ell(t)}\right],
\end{align*} 
Then due to $(2)$ in Lemma~\ref{lemma V derivative}, one can reconstruct the second derivative of $Q$ in the normal direction from
$$
(\ell'(0))^2\; \varepsilon \beta  \;\p_2 \nabla^\perp Q(\0).
$$
Specifically, since $\varepsilon, \beta>0,\ell'(0)>0$ are known, one then recovers $\p_2\nabla^\perp Q(\0)$.

On the other hand, recall that $\nabla Q(x)$ has been determined for any boundary point $x\in\p\Omega$, which is sufficiently close to $p=\0$. Then we can also recover its tangential derivative, that is, 
$
\p_1 \nabla Q(\0).
$
Combining the normal and tangential derivatives together, as this stage, we have recovered
	$$
	\nabla^2 Q(\0).
	$$
Similarly, we can also recover $\nabla^2 Q(x)$ for $x\in\p\Omega$ sufficiently close to $p=\0$, by the same argument above.

\subsubsection{\bf Higher derivatives w.r.t. $t$}

Another direct computation gives that
\begin{align*} 
R^{(4)}(t) &=\hbox{known terms}+\ell'(t)  {\partial X_0 \over \partial \phi} (0, X(\ell(t), \phi(t)))  {d\over dt} \left[ \p^2_t (V - V_0)(X(s,\phi(t))) |_{s=\ell(t)}\right] \\
&\quad+ \ell'(t){\p^3 \over \p t^3}\LC  {\partial X_0 \over \partial \phi} (\ell(t)-s, X(s, \phi(t)))(V - V_0)(X(s,\phi(t))) \RC\Big|_{s=\ell(t)}  \\
& \quad+  \int^{\ell(t)}_0 {\p^4\over \p t^4}\LC  {\partial X_0 \over \partial \phi} (\ell(t)-s, X(s, \phi(t)))(V - V_0)(X(s,\phi(t)))\RC\,ds\\
&=:\hbox{known terms}+R^{(4)}_1(t)+R^{(4)}_2(t)+R^{(4)}_{\text{int}}(t),
\end{align*} 
where $R^{(4)}_{\text{int}}(0)=0$ and these terms $R^{(4)}_1(t), R^{(4)}_2(t)$ are either tangential terms or terms depending only on lower order derivatives of $Q$ ($\nabla^\gamma Q$, $\gamma=1,2$) near $\0$ due to Lemma~\ref{lemma V derivative}.
Therefore, $R^{(4)}$ does not contribute any new information about $Q$, and thus we have to move on to compute $R^{(5)}$.

For $R^{(5)}(t)$, we obtain  
\begin{align*} 
\lim\limits_{t\rightarrow 0} R^{(5)}(t) &=\hbox{known terms}\\
&\quad + \lim\limits_{t\rightarrow 0}\ell'(t) {\partial X_0 \over \partial \phi } (\0, X(\ell(t), \phi(t)))  {d^2\over dt^2 } \left[\p^2_t (V - V_0)(X(s,\phi(t)))  |_{s=\ell(t)}\right] \\
&\quad+  \lim\limits_{t\rightarrow 0}\ell'(t){\p^4 \over \p t^4}\LC  {\partial X_0 \over \partial \phi} (\ell(t)-s, X(s, \phi(t)))(V - V_0)(X(s,\phi(t))) \RC\Big|_{s=\ell(t)}  \\
& \quad+  \lim\limits_{t\rightarrow 0} \int^{\ell(t)}_0 {\p^5\over \p t^5}\LC  {\partial X_0 \over \partial \phi} (\ell(t)-s, X(s, \phi(t)))(V - V_0)(X(s,\phi(t)))\RC\,ds\\
&=:\hbox{known terms}+ \lim\limits_{t\rightarrow 0}R^{(5)}_1(t)+ \lim\limits_{t\rightarrow 0}R^{(5)}_2(t)+ \lim\limits_{t\rightarrow 0}R^{(5)}_{\text{int}}(t).
\end{align*} 
Again, following a similar argument, we only need to focus on the term $ \lim\limits_{t\rightarrow 0}R^{(5)}_1(t)$ since it contains
\begin{align*} 
{d^2\over dt^2 } \left[\p^2_t (V - V_0)(X(s,\phi(t))) |_{s=\ell(t)}\right].
\end{align*} 
In particular, by (2) in Lemma~\ref{lemma V derivative}, we can recover the first component of $\lim\limits_{t\rightarrow 0}R^{(5)}_1(t)$, that is,
$$
    2(\ell'(0))^3\; \varepsilon^2 \beta^2\; \p_2^2\nabla^\perp Q(\0).
$$ 
Note that the tangential derivative $\p_1\nabla\nabla^\perp Q(\0)$ and lower order derivatives $\nabla^\gamma Q(\0),\ \gamma=1,2$ are already known. 
Then $\nabla^3 Q(\0)$ is reconstructed.

Based on the above detailed analysis on the derivative of $R(t)$ up to the $5$-th order, we are ready to show the determination of the boundary jet of $Q$ at the point $p=\0$ by the induction argument.
\begin{theorem}\label{boundary determination in local coordinates}
The measurement $\mathcal S$ in an open neighborhood of $\0$ determines $\p^j_1 \p_2^k Q(\0)$ for any $j, k\in\mathbb Z$, $j+k\geq 1$.
\end{theorem}
\begin{proof}
We have seen that $\p_1^j\p_2^k Q(\0)$ for $j+k\leq 3$ can be determined from $R^{(1)}(t)$, $R^{(3)}(t)$ and $R^{(5)}(t)$ as $t$ goes to $0$.
Given an arbitrary integer $K\geq 3$, by induction, suppose that $\p_1^j\p_2^k Q(\0)$ for $j+k\leq K$, is recovered. Thus $\p_1^j\p_2^k Q(x)$, $j+k\leq K$, is known for $x\in\p\Omega$ sufficiently close to $\0$. This information is enough for determining $\p_1^j\p_2^k Q(\0)$ with $j+k=K+1$, $j\geq 1$.

The only unknown $(K+1)$-th derivative is the term  
$\p_2^{K+1} Q(\0)$. We differentiate $R(t)$, $2K+1$ times, w.r.t. $t$ and obtain
\begin{align*} 
R^{(2K+1)}(t) &=\hbox{known terms}\\
&\quad+\ell'(t) {\partial X_0 \over \partial \phi } (\0, X(\ell(t), \phi(t)))  {d^K\over dt^K} \left[\p^K_t (V - V_0)(X(s,\phi(t)))  |_{s=\ell(t)}\right] \\
&\quad+ \ell'(t){\p^{2K} \over \p t^{2K}}\LC  {\partial X_0 \over \partial \phi} (\ell(t)-s, X(s, \phi(t)))(V - V_0)(X(s,\phi(t))) \RC\Big|_{s=\ell(t)}  \\
& \quad+  \int^{\ell(t)}_0 {\p^{2K+1}\over \p t^{2K+1}}\LC  {\partial X_0 \over \partial \phi} (\ell(t)-s, X(s, \phi(t)))(V - V_0)(X(s,\phi(t)))\RC\,ds.
\end{align*} 
Here those known terms involve the derivatives of $Q$ with orders $\leq K$ and the tangential derivatives of $Q$. 
Therefore based on the analysis above and Lemma~\ref{lemma V derivative}, we finally recover the term 
$$
K!(\ell'(0))^{K+1} \varepsilon^{K} \beta^K\; \p_2^K\nabla^\perp Q(\0),
$$ 
which involves $\p_2^{K+1} Q(\0)$. This completes the proof.
\end{proof}

\subsubsection{\bf Proof of Theorem~\ref{boundary determination}} 
We are ready to show the local result.
\begin{proof}[Proof of Theorem~\ref{boundary determination}]
The proof follows immediately from the discussion above  and Theorem~\ref{boundary determination in local coordinates}.

\end{proof}

\section{Global reconstruction of real-analytic potentials}\label{sec:global reconstruction}
Based on the local result in Section~\ref{sec:local reconstruction}, we can now determine $Q$ globally in $\widetilde{\Omega}$.

\begin{proof}[Proof of Theorem \ref{real analytic case}]
The hypothesis of the theorem gives that $\p\Omega$ is strictly convex at some point $p\in\p\Omega$ w.r.t. the potential $Q$. By Theorem \ref{boundary determination} (also Theorem \ref{boundary determination in local coordinates}), $\p^\alpha Q(p)$ is recovered for any multi-index $\alpha$, $|\alpha|\geq 1$ from the data $\mathcal{S}$. Notice that $\widetilde\Omega$ is an open neighborhood of $\overline\Omega$, then it implies that $p$ is actually an interior point of $\widetilde\Omega$.
Since $Q$ is real-analytic in $\widetilde \Omega$, so is $\nabla Q$. Thus, we can uniquely recover $\nabla Q$ in some neighborhood of $p$. Moreover, since $\widetilde\Omega$ is path-connected and $\nabla Q$ is analytic, we can then determine $\nabla Q$ uniquely in the whole domain $\widetilde{\Omega}$.

Now for any fixed point $x\in\widetilde\Omega$, let $$\gamma: [0,T]\to \overline{\widetilde\Omega},\quad \gamma(0)=z\in \p \widetilde\Omega,\quad \gamma(T)=x$$
be a smooth curve connecting $x\in\widetilde\Omega$ with the boundary of $\widetilde \Omega$ for $T>0$. Notice that $Q$ is given outside $\widetilde\Omega$, thus $Q(z)$ is known for $z\in\p\widetilde{\Omega}$. The Fundamental Theorem of Calculus yields that
$$Q(x)=Q(z)+\int_0^T \left<\nabla Q(\gamma(t)), \dot\gamma(t)\right>\, dt,$$
which then determines $Q(x)$. Here $\left<\cdot,\cdot\right>$ is the Euclidean inner product. Since $x$ is an arbitrary point in $\widetilde\Omega$, this leads to the reconstruction of $Q$ in the whole domain $\widetilde\Omega$. Therefore, the proof of Theorem~\ref{real analytic case} is complete.
\end{proof}

%\newpage
\appendix
\section{Some useful lemmas}\label{sec:appendix}

The following lemmas play an important role in the calculation of $R^{(k)}(t)$ in Section~\ref{sec:local reconstruction}.
Recall that the definition of $V$ and $V_0$ in Proposition~\ref{Prop:SUid} leads to
	$$
	(V-V_0)(X(s,\phi(t))) = \LC\nabla^\perp Q(a_+(s;t)),\,-\nabla^\perp Q(a_-(s;t))\RC^T,
	$$
	and it satisfies the following limit.
	\begin{lemma}\label{lemma limit V}
		For any positive integer $k$, one has
		$$
		\lim\limits_{t\rightarrow 0} \p^k_t(V - V_0)(X(s,\phi(t))) \Big|_{s=\ell(t)}=  
		\left[\begin{array}{c}
		\0\\
		known\\
		\end{array}\right]_{4\times 1},
		$$ 
		where this $\0$ is the zero vector in $\R^2$.
	\end{lemma}
	\begin{proof}
		We first recall that $a_+(0;t)=\0$ and $a_-(0;t)=\xi_t$. For small $s$ and $t$, the asymptotic expansion of the trajectory $a_+(s;t)$ is as follows:
		\begin{align}\label{a+ asymptotic}
		a_+(s;t)=\dot{a}_+(0;t)s+O(s^2).
		\end{align}
	    By taking the $k$-th derivative with respect to $t$, we have
		$$
		\p_t^ka_+(s;t)={d^k\over d t^k} \dot{a}_+(0;t)s+O(s^2),
		$$
		which gives that
		\begin{align}\label{a+ asymptotic diff}
		\lim\limits_{t\rightarrow 0}\p_t^ka_+(s;t)|_{s=\ell(t)} =\lim\limits_{t\rightarrow 0}  {d^k\over d t^k}\dot{a}_+(0;t)\ell(t)+O(\ell(t)^2) =\0 
		\end{align}
		due to $\ell(0)=0$.
	 	Similarly, for $a_-(s;t)$, we have 
	 	$$
	 	\lim\limits_{t\rightarrow 0}\p_t^ka_-(s;t)|_{s=\ell(t)}=\lim\limits_{t\rightarrow 0}{d^k\over dt^k}\xi_t,
	 	$$
	 	which is known.
		Now we turn back to the matrix and obtain
		\begin{align*} 
		\lim\limits_{t\rightarrow 0}\p_t(V - V_0)(X(s,\phi(t)))\Big|_{s=\ell(t)}&=
		\lim\limits_{t\rightarrow 0} \left[ \begin{array}{c}
		\nabla \p_2 Q(a_+(\ell(t);t))\cdot  \p_ta_+(s;t)|_{s=\ell(t)}\\
		-\nabla \p_1 Q(a_+(\ell(t);t))\cdot \p_ta_+(s;t)|_{s=\ell(t)}\\
		-\nabla \p_2 Q(a_-(\ell(t);t))\cdot \p_ta_-(s;t)|_{s=\ell(t)}\\
		\nabla \p_1 Q(a_-(\ell(t);t))\cdot \p_ta_-(s;t)|_{s=\ell(t)}\\
		\end{array}\right]\\
		&=
		\left[\begin{array}{c}
		\0\\
		known\\
		\end{array}\right]_{4\times 1},
		\end{align*}
	    where we used the fact that $\lim\limits_{t\rightarrow 0}\nabla \nabla^\perp Q(a_-(\ell(t);t))=\nabla \nabla^\perp Q(\xi_0)$ is known since $\xi_0\in \R^2\setminus\widetilde\Omega$.
		
		For $k>1$, due to \eqref{a+ asymptotic diff}, we also have  
		$$\lim\limits_{t\rightarrow 0}\p^k_t(\nabla^\perp Q)(a_+(s;t)) |_{s=\ell(t)}=\0 $$ 
		and the term $$	\lim\limits_{t\rightarrow 0}\p^k_t(\nabla^\perp Q)(a_-(s;t)) |_{s=\ell(t)}$$ is known as well,
		which completes the proof.
	\end{proof}
	
Recall that $\dot{a}_+(0;t) = \varepsilon(\alpha(t),\beta t)$ with $\alpha(t)^2+(\beta t)^2=1$ in \eqref{choice of xi(t)}. Then $\alpha(t)=\pm\sqrt{1-(\beta t)^2}$ implies 
$\alpha'(0)=0$. Build upon this and \eqref{a+ asymptotic} and $\ell(0)=0$, we have
\begin{align}\label{dt pt a_+}
\lim\limits_{t\rightarrow 0}{d\over dt}\p_ta_+(s;t)|_{s=\ell(t)}=\lim\limits_{t\rightarrow 0}{d\over dt}\dot{a}_+(0;t)\ell'(t)+\0= \varepsilon (0,\beta) \ell'(0).
\end{align}

In the following lemma, we will only focus on the first component of $(V - V_0)(X(s,\phi(t)))$ since from above discussion we have seen that its second component already contains known data $ (\nabla^\perp Q)(a_-(s;t))|_{s=\ell(t)}$ for sufficiently small $t$.
	\begin{lemma}\label{lemma V derivative} 
		Let $k,\,\eta$ be integers satisfying $k\geq\eta\geq1$.  
		Then the first component of
		$$
		\lim\limits_{t\rightarrow 0}{d^\eta\over dt^\eta} \left[ \p_t^k (V - V_0)(X(s,\phi(t))) |_{s=\ell(t)}\right],
		$$
		that is, 
		$$
		 \mathcal{F}:=\lim\limits_{t\rightarrow 0}{d^\eta\over dt^\eta} \left[ \p_t^k (\nabla^\perp Q)(a_+(s;t))  |_{s=\ell(t)}\right],
		$$
		satisfies the following statements:  
		\begin{enumerate}
			\item  If $k\neq \eta$, then $\mathcal{F}$	
			only depends on the derivatives $\p^{\gamma_1}_1\p^{\gamma_2}_2 Q(\0)$ for $1\leq \gamma_1+\gamma_2\leq \eta+1$.  
			
			\item If $k=\eta$, then $\mathcal{F}$ satisfies
			\begin{align}\label{case k}
             \mathcal{F}
			 = k!  \varepsilon^k (\ell'(0))^{k}  \beta^k \p_2^k\nabla^\perp Q(\0) + \Phi,
			\end{align}
			where the remaining function $\Phi$ only depends on the tangential derivative $$\p_1\nabla^{k-1}\nabla^\perp Q(\0)$$ and lower order derivatives $\p^{\gamma_1}_1\p^{\gamma_2}_2 Q(\0)$ for $1\leq \gamma_1+\gamma_2\leq k$. 
		\end{enumerate} 
	\end{lemma}
	\begin{proof}
		We first consider the case $k=1$ and $\eta=1$ and compute
		$$
		\p_t (\nabla^\perp Q(a_+(s;t))) |_{s=\ell(t)} = \nabla\nabla^\perp Q(a_+(s;t))  \p_ta_+(s;t)|_{s=\ell(t)}.
		$$
		Then from \eqref{a+ asymptotic diff} and \eqref{dt pt a_+}, we have
		\begin{align*} 
		 \lim\limits_{t\rightarrow 0}{d\over dt} \left[ \p_t (\nabla^\perp Q(a_+(s;t))) |_{s=\ell(t)}\right]  
		&= \lim\limits_{t\rightarrow 0}\nabla^2\nabla^\perp Q(a_+(\ell(t);t)) {d\over dt}a_+(\ell(t);t) \p_ta_+(s;t)|_{s=\ell(t)}  \\
		&\quad +\lim\limits_{t\rightarrow 0} \nabla\nabla^\perp Q(a_+(\ell(t);t))  {d \over dt }\p_ta_+(s;t)|_{s=\ell(t)}\\
		&=\varepsilon \ell'(0) \beta \p_2  \nabla^\perp Q(\0),  
		\end{align*}
		which shows \eqref{case k} for $k=1$.
		
		Next we consider the case $k=2$ and then take $1\leq \eta \leq 2$. 
		Note that
		\begin{align*} 
		\p^2_t (\nabla^\perp Q(a_+(s;t))) |_{s=\ell(t)} &= \nabla^2\nabla^\perp Q(a_+(s;t))  (\p_ta_+(s;t))^2|_{s=\ell(t)}\\
		&\quad +\nabla\nabla^\perp Q(a_+(s;t))\p_t^2a_+(s;t)|_{s=\ell(t)}.
		\end{align*}		
		For $k=2$ and $\eta=1$, a direct computation yields
		\begin{align*} 
		\lim\limits_{t\rightarrow 0}{d \over dt } \left[ \p^2_t (\nabla^\perp Q(a_+(s;t))) |_{s=\ell(t)}\right]  = \nabla\nabla^\perp Q(\0) \lim\limits_{t\rightarrow 0}{d\over dt} \p_t^2a_+(s;t)|_{s=\ell(t)},
		\end{align*}
		which only contains $\p_1^{\gamma_1}\p_2^{\gamma_2}Q(\0)$ for $1\leq \gamma_1+\gamma_2\leq 2$. Similarly, for $k=2$ and $\eta=2$, we have
		\begin{align*} 
		&\lim\limits_{t\rightarrow 0}{d^2 \over dt^2 } \left[ \p^2_t (\nabla^\perp Q(a_+(s;t))) |_{s=\ell(t)}\right] \\
		&= 2\lim\limits_{t\rightarrow 0} \nabla^2\nabla^\perp Q(a_+(\ell(t);t))  \LC{d \over dt }\p_ta_+(s;t)|_{s=\ell(t)}\RC^2\\
		&\quad + \hbox{tangential derivative } \p_1  \nabla\nabla^\perp Q(\0)+\hbox{lower oder derivatives}\\
		&= 2 \varepsilon^2  (\ell'(0))^2 \beta^2\p_2^2  \nabla^\perp Q(\0)\\
		&\quad + \hbox{tangential derivative } \p_1  \nabla\nabla^\perp Q(\0)+\hbox{lower oder derivatives},
		\end{align*}
		where we used \eqref{dt pt a_+} ($\lim\limits_{t\rightarrow 0}  {d \over dt }\p_ta^1_+(s;t)|_{s=\ell(t)}=0$) in the last identity. So far we have shown (1) and (2) for $k=2$.

		Now for $k>2$, we have
		\begin{align*} 
		&\p^k_t (\nabla^\perp Q(a_+(s;t))) |_{s=\ell(t)} \\
		&= \nabla^k \nabla^\perp Q(a_+(s;t))  (\p_t a_+(s;t))^k|_{s=\ell(t)} +\ldots 
	       +\nabla\nabla^\perp Q(a_+(s;t))\p^k_t  a_+(s;t)|_{s=\ell(t)}.
		\end{align*}	
		Along the lines of similar arguments for the case $k=1,2$ above, we have
		\begin{align*} 
		 \lim\limits_{t\rightarrow 0}{d^\eta \over dt^\eta } \left[ \p^k_t (\nabla^\perp Q(a_+(s;t))) |_{s=\ell(t)}\right] \qquad \hbox{for }1\leq\eta< k
		\end{align*}
		depending only on the lower derivatives $\nabla^{\gamma}\nabla^\perp Q(\0)$ for $1\leq \gamma\leq \eta$.
		Finally for $\eta=k>2$, we get
		\begin{align*} 
		&\lim\limits_{t\rightarrow 0}{d^k\over dt^k } \left[ \p^k_t (\nabla^\perp Q(a_+(s;t))) |_{s=\ell(t)}\right]\\
		&=k! \lim\limits_{t\rightarrow 0} \nabla^k\nabla^\perp Q(a_+(\ell(t);t))  \LC{d \over dt }\p_ta_+(s;t)|_{s=\ell(t)}\RC^k\\
		&\quad + \hbox{tangential derivative } \p_1  \nabla^{k-1}\nabla^\perp Q(\0)+\hbox{lower oder derivatives}, 
		\end{align*}
		and then we apply \eqref{a+ asymptotic diff} again. This completes the proof of the lemma.
	\end{proof}

\vskip1cm
\noindent \textbf{Acknowledgment.}
R.-Y. Lai is partially supported by the NSF grant DMS-1714490.

\vskip1cm

\bibliographystyle{plain}
\bibliography{dipolebib}

\end{document}